\documentclass[12pt]{amsart}
\usepackage{graphicx}
\usepackage{a4wide}
\usepackage{amssymb}
\usepackage{amsthm}
\usepackage{amsmath}
\usepackage{amscd}
\usepackage{verbatim}
\usepackage[mathscr]{eucal}
\usepackage[all]{xy}
\theoremstyle{plain}
\newtheorem{theorem}{Theorem}[section]

\newtheorem{lemma}[theorem]{Lemma}

\theoremstyle{definition}

\newtheorem{example}[theorem]{Example}
\newtheorem{remark}[theorem]{Remark}

\theoremstyle{remark}

\newcommand{\calT}{\mathcal{T}}

\newcommand{\val}{\textrm{val}}

\newcommand{\frakp}{\mathfrak{p}}

\newcommand{\GL}{\mathbf{GL}} 
\newcommand{\pitilde}{\tilde{\pi}}

\newcommand{\Z}{\mathbb{Z}}

\newcommand{\C}{\mathbb{C}}

\newcommand{\F}{\mathbb{F}}

\newcommand{\ds}{\displaystyle}
\begin{document}

\title[ Formulas and vanishing conditions for   coefficients of eigenforms]
{Explicit formulas and vanishing conditions for certain coefficients of Drinfeld-Goss Hecke eigenforms}

\author{Ahmad El-Guindy}
\address{Current address: Science Program, Texas A\&M University in Qatar, Doha, Qatar}
\address{Permanent address: Department of Mathematics, Faculty of Science, Cairo University, Giza, Egypt 12613}
\email{a.elguindy@gmail.com}

\keywords{Drinfeld-Goss modular forms, Hecke operators, recurrence relations}
\subjclass[2010]{11F52, 11F25}
\thanks{}
\date{}
\begin{center}
 \emph{Dedicated to the  memory of David Goss} \\ \smallskip
    
\end{center}
\begin{abstract}
We obtain a closed form polynomial expression for certain coefficients of Drinfeld-Goss double-cuspidal modular forms which are eigenforms for the degree one Hecke operators with power eigenvalues, and we use those formulas to prove vanishing results for an infinite family of those coefficients.
\end{abstract}

\maketitle

\section{Introduction}
 In the classical theory of modular forms, a central role is played by Hecke operators and their eigenforms. The Fourier coefficients of those Hecke eigenforms satisfy elegant relations which encode important number theoretic information. In the setting of Drinfeld modules, there is a parallel theory of modular forms and Hecke operators which was introduced by  Goss \cite{Goss1, Goss2} and expanded on by Gekeler \cite{Gekeler88} and subsequently studied by many authors. While the basic definitions mirror the classical case in a natural way, there are some significant differences between the nature of the Hecke action in the classical vs. Drinfeld setting, and it is of great interest to unravel the mysteries of the Hecke action in the latter case. 

One example of a phenomena that exists exclusively in the Drinfeld setting is the fact that the same system of eigenvalues might appear for different eigenforms. In order to formulate this more precisely, we need to recall some basic definition and fix some notation.  Let $q := p^r$ be a power of a prime $p$, and consider the finite field $\F_q$.  Let $A := \F_q[\theta]$, $K := \F_q(\theta)$. We write $A_+$ for the set of monic polynomials in $A$.

For $a \in A$,  we define $|a| := q^{\deg a}$ and extend this absolute value to~$K$. The completion of $K$ with respect to this absolute value, denoted by $K_\infty$, is given by $K_\infty=K((\theta^{-1}))$. Let $\C_\infty$ be the completion of a fixed algebraic closure of $K_\infty$. The field $\C_\infty$ is complete and algebraically closed. The rigid analytic space $\Omega:=\C_\infty\setminus K_\infty$ is the function field analogue of the complex upper half-plane, and we shall refer to it as the \emph{Drinfeld upper half-plane}. A holomorphic function $f:\Omega\rightarrow \C_\infty$ is called a (Drinfeld-Goss) modular form of weight $k$ (a positive integer) and type $m$ (a residue class in $\Z/(q-1)$) if for every $\gamma=\left(\begin{smallmatrix} a&b\\c&d\end{smallmatrix}\right) \in \GL_2(A)$ we have
\[
f(\gamma z)=(cz+d)^k (\det \gamma)^{-m}f(z),
\]
in addition to a holomorphicity condition at the `infinite cusp' (see \cite{Gekeler88} for the precise definitions). This is clearly strongly analogous to the classical setting, and similar to that case the `holomorphicity at the cusp' implies a series expansion in powers of a certain `uniformizer at the cusp'. In the classical setting the uniformizer is $\exp(2\pi i z)$ whereas in the Drinfeld setting that role is played by $t(z)$ given by

\[
	t(z) := \frac{1}{\pitilde} \sum_{a \in A} \frac{1}{z + a},
\]
where $\pitilde$ is a fixed choice of a fundamental period of the Carlitz module. Let $M_{k,m}$ denote the vector space of modular forms of weight $k$ and type $m$. For any $f\in M_{k,m}$ and $z\in \C_\infty$ with  $|z|_i:= \inf_{a \in K_\infty} |z - a|$ sufficiently large ($|z|_i$ is a function field analogue of the `imaginary distance'), then we have 
\begin{equation}\label{uexp}
	f(z) = \sum_{n = 0}^\infty a_n t(z)^n , \qquad \qquad a_n \in \C_\infty,
\end{equation}
and this expansion determines $f$ uniquely. If $f \in M_{k, m}$ and $a_0 = 0$, then $f$ is called \emph{cuspidal}, and if both $a_0=0$ and $a_1=0$ we call $f$ \emph{double-cuspidal}. We shall denote the subspace of cuspidal forms inside $M_{k, m}$ by~$S_{k, m}$, and the space of double-cuspidal forms by~$M^2_{k,m}$. Some examples of modular forms are the weight $k(q-1)$ type $0$ Eisenstein series $E_{k(q-1)}$ ($k\geq 1$)and the weight $q+1$ type $1$ cusp form $h$, which are given respectively by
\begin{equation}\label{hdef}
E_{k(q-1)}(z)=\ds \sum_{(a,b)\in A^2-(0,0)} \frac{1}{(az+b)^{k(q-1)}}, \, \, \, h(z)=\sum_{a\in A_+}a^{q}t(az).
\end{equation}
It is well known that the algebra of modular forms of any weight and type is generated by polynomials in $E_{q-1}$ and $h$. Furthermore, the subalgebra of type $0$ forms is generated by $E_{q-1}$ and the Delta function (normalized to have leading coefficient $1$) $\Delta:=h^{q-1}$.

Let $\frakp$ be a monic prime  in $A$. The $\frakp^\text{th}$ Hecke operator of weight~$k$  on $M_{k,m}$, denoted by $\calT_{\frakp, k}$, is defined by (see \cite[\S 7]{Gekeler88}):
\[
\calT_{\frakp,k}f(z)=\frakp^{k}f(\frakp z)+\sum_{b\in A, \deg(b)<\deg(\frakp)}f\left(\frac{z+b}{\frakp}\right).
\]
This again is reasonably analogous to the definition of Hecke operators in the classical theory. The Hecke action preserves $M_{k,m}$ as well as the subspaces of cuspidal and double-cuspidal forms. The repeated eigensystem phenomena mentioned above becomes evident with the first few computations of eigenforms as it was shown by Goss \cite{Goss1} that $\calT_\frakp(E_{q-1})=\frakp^{q-1}E_{q-1}$ and $\calT_\frakp(\Delta)=\frakp^{q-1}\Delta$ for all prime $\frakp \in A_+$ (we shall customarily drop the weight $k$ from the notation when it is clear from the context). Indeed, using the concept of $A$-expansion to define cusp forms, Petrov was able to construct infinite families of Hecke eigenforms with the same eigensystem. Namely, it follows from \cite[Theorem 1.3 and Theorem 2.3]{Pet} that if $k$ and $n$ are positive integers for which $k-2n$ is a positive multiple of $(q-1)$ and also $n\leq p^{\val_p(k-n)}$, then 
\begin{equation}\label{Aexp}
f_{k,n}(z):=\sum_{a\in A_+}a^{k-n}G_n(t(az))
\end{equation}
satisfies $\calT_\frakp(f_{k,n})=\frakp^n f_{k,n}$ for all monic prime $\frakp$. It thus follows that for any $n\geq 1$ there are infinitely many eigenforms (in different weights) with eigensystem $\frakp^n$. One natural question which we aim to address in this work is to identify common features of forms with such common eigensystems, in particular in connection with their $t$-expansions. The tools and results developed here naturally lead to interesting new results on the vanishing and non-vanishing of coefficients of eigenforms. 

Such questions of vanishing and non-vanishing have been an important topic of study in the classical theory of modular forms (see \cite{BOR, Lehmer, Ono} for example). They were addressed in the Drinfeld setting for certain special modular forms \cite{Gekeler88, Gekeler99, Arm, BacLop}. For instance, Gekeler \cite{Gekeler88} shows that if $\sum a_is^i$ is the expansion of the forms $\Delta$ or $g_{q^k-1}$ (where $g_{q^k-1}$ is the constant multiple of $E_{q^k-1}$ normalized to have leading coefficient $1$)  with respect to $s=t^{q-1}$ then $a_i\neq 0$ implies $i\equiv 0,1 \pmod q$. In \cite{Gekeler99},  Gekeler also shows (among other things, as well as similar results for $g_{q^k-1}$ and $h$) that for $\Delta$ one has $\deg(a_{1+i})=i$ if and only $i\equiv 0,q \pmod {q^2}$, which of course implies non-vanishing for the coefficients in those classes. The proofs rely on special properties of $\Delta$ (such as its product expansion and explicit representation in terms of Eisenstein series). Those properties don't always carry to general Hecke eigenforms (cf. Section \ref{examples} below), nonetheless our results provide information about an infinite family of nontrivial $t$-expansion coefficients of any Hecke eigenforms with power eigensystems. This includes all the eigenforms with $A$-expansions, but it extends beyond that as well (again see Section \ref{examples} below for illustrations.)

In \cite{Arm} Armana had provided formulas for type $0$ and type $1$ cusp forms (with level) in terms of power sums of coefficients of the Carlitz module. Those results were extended by Baca and Lopez in \cite{BacLop} by explicitly evaluating those sums (together with some of Gekeler's results from \cite{Gekeler88, Gekeler99}) to obtain explict formulas for some of the coefficients of the forms $h, \Delta$ and $g_{q^k-1}$.
Work of the author and Petrov \cite{ElgPet} provides another angle on Armana's work, as it was shown that for any type level $1$ eigenform with power eigensystem, there is an infinite family of coefficients (coinciding with the ones given by Armana for the type $0$ and type $1$ cases) which are completely determined by the corresponding eigenvalues. One of the purposes of the present work is to make those results more precise by providing explicit formulas for those coefficients and use that to concretely answer questions of their vanishing and nonvanishing. 

In the next section we shall state our main results, throughout  assuming $q=p$ is prime for simplicity (this was also assumed in \cite{Arm, ElgPet}, although some results hold in more generality). In Section 3 we study a certain recurrence relation relevant to the Hecke action and prove existence and uniqueness results for its solutions possessing certain natural symmetries. We use those results to prove the main theorem in Section 4 and also completely determine the cases in which the coefficients under study vanish. We conclude with some concrete examples in Section 5.


\section*{acknowledgement}
This paper is dedicated to the memory of David Goss both for his profound and lasting contribution to the subject as well as his constant support and encouragement for those working in it. His deep knowledge as well as his kindness and generosity will truly be missed.


\section{Statement of results}
 In order to motivate the choice of coefficients we shall focus on, we start by recalling the formula for the  Hecke action on the $t$-expansion from \cite{Gekeler88}

\begin{equation} \label{coeffaction}
 \calT_{\frakp, k} \left ( \sum_{n = 0}^\infty a_n t^n \right) = \frakp^{k} \sum_{n = 0}^\infty a_n t_\frakp^n +  \sum_{n=0}^\infty a_n G_{n, \frakp} (\frakp t),
\end{equation}
where $G_{n, \frakp} (X)$ is the $n$-th Goss polynomial of the finite lattice formed by the $\frakp$-torsion of the Carlitz module (see \cite[(3.4)]{Gekeler88} for the definition of Goss polynomials of a lattice). For a monic prime $\frakp$ of degree $1$, Gekeler  uses \eqref{coeffaction} to obtain 
\begin{equation}\label{hecke1}
\begin{array}{ll}
a_n(\calT_{\frakp,k} f)&=\ds \frakp^k \left(\sum_{j, s\geq 0, j q+s(q-1)=n}(-1)^s \binom{j+s-1}{s}\frakp^sa_j\right)\\
&\ds+\sum_{i=0}^{n-1}\binom{n-1}{i}\frakp^{n-i}a_{n+i(q-1)}.
\end{array}
\end{equation}
There is a clear difference in complexity between \eqref{hecke1} and the corresponding formula in the classical setting (namely (slightly abusing notation)  $a_n(\calT_{p,k}f)=a_{pn}+p^{k-1} a_{n/p}$, where the latter term is $0$ if $p\nmid n$). That classical formula implies the well-known fact that  the $p^\textrm{th}$ coefficient of the Fourier expansion of a normalized Hecke eigenform is actually given by the eigenvalue corresponding to the action of $\calT_p$ on that form, whereas in the Drinfeld setting there is no general clear connection between the eigenvalues and the  coefficients of the cuspidal expansions. Nonetheless, our results will enable precise computation of a certain family of coefficients of Drinfeld-Goss Hecke eigenforms. 
Given any integer $n\geq 2$ we can write 
\begin{equation}\label{nrep}
n=1+q^{\nu_1}+q^{\nu_2}+\dots+q^{\nu_\ell}
\end{equation}
 with integers $\nu_i \geq 0$. The correspondence between $n$ and the length $\ell$ multiset (i.e. elements may repeat) $\nu=\{\nu_1,\dots, \nu_\ell\}$ is unique.  Let  $V_\ell:=\{\nu=\{\nu_1,\dots,\nu_\ell): \nu_i\geq 0\}$ be the (infinite) set of all such multisets. For $\nu\in V_\ell$ we set $q^\nu$ to be the integer $q^{\nu_1}+\dots+q^{\nu_\ell}$. Also, given a set $I\subset I_\ell:=\{1,2,\dots,\ell\}$ with complement $I^c\subset I_\ell$, we write $\widehat{\nu}(I):=\{\nu_j: j \in I^c\}$ and
\[
 \nu^+(I):=\{1+\nu_i: i\in I\}\cup \widehat{\nu}(I).
\]
(Both $\widehat{\nu}(I)$ and $\nu^+(I)$ are multisets in general).
Note that $\nu^+(\emptyset)=\widehat{\nu}(\emptyset)=\nu$, and that $|\widehat{\nu}(I)|=\ell-|I|$ whereas $\nu^+(I)\in V_\ell$ whenever $\nu\in V_\ell$. It turns out that for $\ell\leq q-1$, the following useful simplification of  \eqref{hecke1} can be obtained for the family of coefficients indexed by $V_\ell$.

\begin{lemma}\cite[Lemma 5.1]{ElgPet}
For $f=\sum a_i t^i\in M^2_{k,m}$ and $\nu=\{\nu_1,\dots,\nu_\ell\}$ a multiset of nonnegative integers of length $\ell\leq q-1$ we have
\begin{equation}\label{truemaster1}
\begin{array}{ll}
\ds a_{1+q^\nu}(\calT_{\theta,k} f)&\ds=\sum_{i=0}^{q^\nu}\binom{q^\nu}{i}\theta^{1+q^\nu-i}a_{1+q^\nu+i(q-1)}\\
&=\ds \sum_{I\subset\{1,\dots,\ell\}}{\theta^{1+q^{\widehat{\nu}(I)}}}a_{1+q^{\nu^+(I)}} 
\end{array}
\end{equation}
\end{lemma}
\begin{remark}
 We take this opportunity to correct a typographical error in the statement of the above lemma in \cite{ElgPet} where a factor of $\binom{q^\nu}{q^{\nu(I)}}$ erroneously appeared in the second summand. The binomial coefficient in the first summand is correct however, and, when needed, it is accounted for in the second summand not by a binomial coefficient but rather by repetitions in some of the entries of $\nu$ resulting in a number of subsets of $I_\ell$ yielding the same $\nu^+(I)$. Other than the need to remove the factors $\binom{q^\nu}{q^{\nu(I)}}$ from equations (22) and (23) of \cite{ElgPet}, that error has no consequences on the results of that paper (which were qualitative in nature).
\end{remark}

\begin{theorem}\label{modularmain}
Assume that $f=\sum_{i=2}^\infty a_i t^i\in M^2_{k,m}$ satisfies  $\calT_{\theta,k}(f)=\theta^{1+q^{N_1}+\dots+q^{N_\ell}}f$ with   $1\leq \ell\leq q-1$ and $N_i\geq 0$. For each $\nu\in V_\ell$ let \begin{equation}\label{blocks}
B(\nu):=\prod_{i=1}^\ell\prod_{j=0}^{\nu_i-1} (\theta^{q^{N_i}}-\theta^{q^j})
\end{equation} 
and consider the action of $S_\ell$ on $B(\nu)$ given by permuting the $N_i$; namely

\begin{equation}\label{sigmaB}
B^\sigma(\nu):=\prod_{i=1}^\ell\prod_{j=0}^{\nu_i-1} (\theta^{q^{N_{\sigma(i)}}}-\theta^{q^j}).
\end{equation} 
Then for all $\nu\in V_\ell$ we have
\begin{equation}\label{mainf}
\ell! a_{1+q^\nu}= a_{1+\ell}\sum_{\sigma \in S_\ell} B^\sigma(\nu).
\end{equation}
\end{theorem}
We list a few examples to illustrate the theorem.

\begin{example}\label{ex1}
\begin{enumerate}
\item If $\calT_{\theta,k}(f)=\theta^{1+q^{N_1}}$, then we are in the case $\ell=1$ of Theorem \ref{modularmain}. Thus either $a_{1+q^i}=0$ for all $i\geq 0$ or we might normalize to have $a_2=1$ and 
\[
a_{1+q^i}=\begin{cases}
(\theta^{q^{N_1}}-\theta)(\theta^{q^{N_1}}-\theta^q)\cdots(\theta^{q^{N_1}}-\theta^{q^{(i-1)}}) \textrm{ if }  1\leq i\leq N_1,\\
0 \textrm{ if } i >N_1.
\end{cases}
\] 

\item If $\calT_{\theta,k}(f)=\theta^{1+q^{N_1}+q^{N_2}}$, then we are in the case $\ell=2$ of Theorem \ref{modularmain}. Thus either $a_{1+q^{i_1}+q^{i_2}}=0$ for all $i_1, i_2\geq 0$ or we might normalize to have $a_3=1$ and 
\[
\begin{array}{ll}
2 a_{2+q}&=(\theta^{q^{N_1}}-\theta)+(\theta^{q^{N_2}}-\theta)\\
a_{1+2q}&=(\theta^{q^{N_1}}-\theta)(\theta^{q^{N_2}}-\theta)\\
2a_{2+q^2}&=(\theta^{q^{N_1}}-\theta)(\theta^{q^{N_1}}-\theta^q)+(\theta^{q^{N_2}}-\theta)(\theta^{q^{N_2}}-\theta^q)\\
2a_{1+q+q^2}&=(\theta^{q^{N_1}}-\theta)(\theta^{q^{N_1}}-\theta^q)(\theta^{q^{N_2}}-\theta)+(\theta^{q^{N_2}}-\theta)(\theta^{q^{N_2}}-\theta^q)(\theta^{q^{N_1}}-\theta)\\
a_{1+2q^2}&=(\theta^{q^{N_1}}-\theta)(\theta^{q^{N_1}}-\theta^q)(\theta^{q^{N_2}}-\theta)(\theta^{q^{N_2}}-\theta^q)\\
&\vdots\\
a_{1+q^{i_1}+q^{i_2}}&=0 \textrm{ whenever } \max(i_1,i_2) >\max(N_1,N_2) \textrm{ or } \min(i_1,i_2)> \min(N_1,N_2). 
\end{array}
\]

\end{enumerate}
\end{example}

\section{Existence and uniqueness of universal recurrence solutions}
In this section, we study a ``universal" form of the recurrence \eqref{truemaster1} in which we replace the power eigenvalue on the left hand side by a product of generic variables. Namely let $\ell$ be an integer satisfying  $1\leq \ell \leq q-1$ and let $x_1,\dots, x_\ell$ be a collection of variables. Consider the recurrence 
\begin{equation}\label{newmaster}
b_{q^\nu}\prod_{i=1}^\ell x_i=\sum_{I\subset\{1,\dots,\ell\}} b_{q^{\nu^+(I)}}\prod_{j\in I^c}\theta^{q^{\nu_j}}
\end{equation} 
where $I^c$ denotes the set-theoretic complement of $I$ in $\{1,\dots, \ell\}$ and $\nu$ runs over all multisets in $V_\ell$. In general such a recurrence could have many solutions, but we will show that, when a certain natural condition is satisfied, there is a unique solution with a given initial value, and we also exhibit explicit formulas for that solution. Our interest will be in solutions to the recurrence with a certain symmetry. Namely, set $Y:=\F_q[x_1,\dots,x_\ell, \theta]$ and consider the operator $D_1$ on $Y$ defined by
\begin{equation}\label{D1def}
D_1f(x_1,\dots,x_\ell,\theta)=f(x_1+1,x_2+1,\dots,x_\ell+1,\theta+1)-f(x_1,x_2,\dots,x_\ell,\theta).
\end{equation}
Furthermore, let $D_j:=D_1^j$ for $0\leq j\leq q-1$. We shall call an element $y$ of $Y$ \emph{translation invariant} if $D_1(y)=0$. The set of translation invariant elements forms an $\F_q$ sub-algebra of $Y$. We shall call a sequence of elements in $Y$ translation invariant if each member is translation invariant.

\begin{theorem}[Uniqueness]\label{unq}
Let $b_{q^\nu}$ and $c_{q^\nu}$ be two translation invariant sequences in $Y$ indexed by $V_\ell$  such that both satisfy \eqref{newmaster}. If $b_\ell=c_\ell$ (which corresponds to $\nu=\{0,0,\dots,0\} \in V_\ell$) then $b_{q^\nu}=c_{q^\nu}$ for all $\nu \in V_\ell$. 
\end{theorem}

\begin{proof}
Without loss of generality we can order the entries of any $\nu\in V_\ell$ in non-ascending order. Furthermore, we can define lexicographical order on $V_\ell$ in the usual way. It is easy to verify that this defines a total order on $V_\ell$. If the theorem is not true then there must be a minimal multiset $\nu$ for which $b_{q^\nu}\neq c_{q^\nu}$. We can't have $\nu=\{0,\dots,0\}$ since $b_\ell=c_\ell$ is given. Thus at least one entry of $\nu$ is nonzero. Define a multiset $\mu\in V_\ell$ by 
\[
\mu_j=\begin{cases}
\nu_j-1 \textrm{ if } \nu_j> 0, \\
0 \textrm{ if } \nu_j=0.  
\end{cases}
\]
Let $e$ and $e^\prime$ denote the number of $0$ entries in $\nu$ and $\mu$, respectively. Clearly $e^\prime \geq e$ and  $\nu=\mu^+(\{1,2,\dots, \ell-e\})$. In fact, we have $\mu^+(I)=\nu$ exactly when $I=\{1,2,\dots, \ell-e^\prime\}\cup J$, where $J$ is any subset with exactly  $e^\prime-e$ elements of $\{\ell-e^\prime+1, \ell-e^\prime+2,\dots, \ell\}$.  Furthermore if $I\subset\{1,2,\dots,\ell\}$ satisfies $|I|\leq \ell-e$ then $\mu^+(I)\leq \nu$ in lexicographical order. (This is clear for $I\subset\{1,2,\dots, \ell-e\}$. If $I$ contains entries $j$ larger than $\ell-e$ then it will have to miss an equal number of entries $i_j\leq \ell-e$ and we have $1=\mu_j+1\leq \nu_{i_j}$). It is easy to verify that if $|I^c|<e$ then $D_e(\prod_{j\in I^c}\theta^{q^{\nu_j}})=0$, whereas for $|I^c|=e$ we have
$D_e(\prod_{j\in I^c}\theta^{q^{\nu_j}})=e!$. Thus applying $D_e$ to \eqref{newmaster} annihilates all terms with $|I|>\ell-e$. Among the remaining terms there are exactly $\binom{e^\prime}{e}$ terms with $\mu^+(I)=\nu$ (all with $|I|=\ell-e$). We thus have 

\begin{equation}\label{e}
b_{q^\mu}D_e\left(\prod_{i=1}^\ell x_i\right)-\sum_{I\subset\{1,\dots,\ell\}, |I|\leq\ell-e,\mu^+(I)\neq\nu} b_{q^{\mu^+(I)}}D_e\left(\prod_{j\in I^c}\theta^{q^{\nu_j}}\right)=\binom{e^\prime}{e}e!b_{q^\nu}.
\end{equation} 
By the minimality of $\nu$, we have $b_{\mu^+(I)}=c_{\mu^+(I)}$ for all the subsets on the left hand side (including $\mu=\mu^+(\emptyset)$), and since $\binom{e^\prime}{e}e!\neq 0$  we  obtain $b_{q^\nu}=c_{q^\nu}$;  a contradiction which proves the theorem.
\end{proof}

Our next goal is to present explicit solutions to the recurrence \eqref{newmaster}. To achieve that, consider  the action of the symmetric group $S_\ell$ on  $Y:=\F_q[x_1,\dots,x_\ell,\theta]$ given by $\sigma(x_i)=x_{\sigma(i)}$,  $\sigma(\theta)=\theta$,  and extended naturally to all of $Y$. We shall also need the following lemma.

\begin{lemma}\label{newkey}
Let $x_1,\dots, x_\ell$ and $t_1,\dots, t_\ell$ be any collections of variables, then the following identity holds

\begin{equation}\label{xt}
\prod_{i=1}^\ell x_i=\sum_{I\subset\{1,2,\dots, \ell\}} \left[\prod_{i\in I} (x_i-t_i)  \prod _{j\in I^c} t_j\right].
\end{equation}

\end{lemma}

\begin{proof}
We shall proceed by induction on $\ell$. For $\ell=1$ we clearly have $x_1=(x_1-t_1)+t_1$, where the first summand corresponds to $I=\{1\}$ and the second to $I=\emptyset$. Assuming the truth of the statement for $\ell$, write
\[
\begin{split}
\prod_{i=1}^{\ell+1}x_i &=[(x_{\ell+1}-t_{\ell+1})+t_{\ell+1}]\prod_{i=1}^\ell x_i\\
&=\sum_{I\subset\{1,\dots,\ell+1\}, \ell+1 \in I}\left[\prod_{i\in I} (x_i-t_i)  \prod _{j\in I^c} t_j\right]+ \sum_{I\subset \{1,\dots,  \ell+1\}, \ell+1\in I^c}\left[\prod_{i\in I} (x_i-t_i)  \prod _{j\in I^c} t_j\right]\\
&=\sum_{I\subset \{1,\dots, \ell+1\}\in I}\left[\prod_{i\in I} (x_i-t_i)  \prod _{j\in I^c} t_j\right]
\end{split}
\]
The last equality, which completes the induction and establishes the result, follows since the conditions $\ell+1\in I$ and $\ell+1\in I^c$ partition the subsets of $\{1,\dots, \ell+1\}$ into two non-overlapping collections.
\end{proof}

\begin{theorem}[Explicit formulas]
Assume the notation above, and for any multiset $\nu\in V_\ell$ set 
\begin{equation}\label{recsol1}
P(\nu):=\prod_{i=1}^\ell\prod_{j=0}^{\nu_i-1} (x_i-\theta^{q^j}).
\end{equation}   Then the sequence given for any $\nu \in V_\ell$ by 
\begin{equation}\label{formula0}
b_{q^\nu}=\frac{1}{\ell!}\sum_{\sigma \in S_\ell} \sigma(P(\nu))
\end{equation}
is a solution to \eqref{newmaster} which satisfies $b_\ell=1$.
\end{theorem}

\begin{proof}
We start by noting that for any $I\subset\{1,\dots,\ell\}$ we have
\[
P(\nu^+(I))=P(\nu)\prod_{i\in I}(x_i-\theta^{q^{\nu_i}})
\]

The following verification shows that $b_{q^\nu}$ satisfies \eqref{newmaster}
\[
\begin{array}{ll}
\ds \ell!\sum_{I\subset\{1,\dots,\ell\}} b_{q^{\nu^+(I)}}\prod_{j\in I^c}\theta^{q^{\nu_j}}&=\ds\sum_{\sigma\in S_\ell}\sum_{I\subset\{1,\dots,\ell\}}\prod_{j\in I^c}\theta^{q^{\nu_j}}\sigma (P(\nu^+(I)))\\
&=\ds\sum_{\sigma\in S_\ell}\sigma\left(\sum_{I\subset\{1,\dots,\ell\}}\prod_{j\in I^c}\theta^{q^{\nu_j}} P(\nu^+(I))\right)\\
&=\ds\sum_{\sigma\in S_\ell}\sigma\left(P(\nu)\sum_{I\subset\{1,\dots,\ell\}}\prod_{j\in I^c}\theta^{q^{\nu_j}} \prod_{i\in I}(x_i-\theta^{q^{\nu_i}})\right)\\
&=\ds \prod_{i=1}^\ell x_i\sum_{\sigma\in S_\ell}\sigma\left(P(\nu)\right)=\ell! b_{q^\nu}\prod_{i=1}^\ell x_i, \\
\end{array}
\] 
where the penultimate equality follows from Lemma \ref{newkey} and the fact that $\sigma\left(P(\nu)\prod_{i=1}^\ell x_i\right)=\left(\prod_{i=1}^\ell x_i\right)\sigma(P(\nu))$ since $\prod_{i=1}^\ell x_i$ is invariant under $S_\ell$. This completes the proof.
\end{proof}

\section{Proof of Theorem \ref{modularmain} and consequences}

\begin{proof}[Proof of Theorem \ref{modularmain}] It is clear that \eqref{newmaster} and the quantities $P(\nu)$ specialize to \eqref{truemaster1} and $B(\nu)$, respectively,  upon setting  $x_i=\theta^{q^{N_i}}$ for $1\leq i\leq \ell$ and substituting $a_{1+q^\nu}$ for $b_{q^\nu}$. It thus follows that we can obtain \eqref{mainf} from \eqref{formula0} once we can establish the translational invariance property that we used in Theorem \ref{unq}. This can be deduced by noting that if we replace $\theta$ by a translate $\tilde{\theta}=\theta+c$ with $c\in\F_q$ as a generator of $A$ over $\F_q$, then the corresponding parameter $\tilde{t}$ at the cusp  will be equal to $t$. It also follows that if $\calT_{\theta,k} f=\theta^N f$ then $\calT_{\tilde{\theta},k} f=\tilde{\theta}^N f$, and hence the coefficients of $f$ must be invariant under $\theta\mapsto \tilde{\theta}$, which implies that $D_1(a_i)=0$ for all coefficients of $f$. 
\end{proof}


Using formula \eqref{mainf} as well as the tools developed in \S 3, we can now state the following vanishing criteria which generalizes what was observed in the examples of \S 2.

\begin{theorem}
Let $f\in M_{k,m}^2$ be as in Theorem \ref{modularmain}. Then we have $a_{1+q^\nu}=0$ if and only if  $N_i<\nu_i$ for some $1\leq i\leq \ell$.
\end{theorem}
\begin{proof}
First, we prove that the condition $\nu_i>N_i$ implies vanishing. For such $\nu$, we clearly have $B(\nu)=0$. We would like to also show that $B^\sigma(\nu)=0$ for all $\sigma\in S_\ell$, which would follow if we can show that for any such $\sigma$ there exists an index $j\in\{1,\dots,\ell\}$ for which $\nu_j >N_{\sigma(j)}$. Indeed we either have that $\sigma$ permutes the elements of $\{1,\dots, i-1\}$ among themselves, in which case we must have $\sigma(i)\geq i$ and hence $N_{\sigma(i)}\leq N_i< \nu_i$, or else there is a value $j\in\{1,\dots,i-1\}$ with $\sigma(j)\geq i$, again yielding $N_{\sigma(j)}\leq N_i<\nu_i\leq\nu_j$, establishing vanishing. To prove necessity, we note that if $\nu_i\leq N_i$ for all $1\leq i\leq \ell$, then clearly $B(\nu)\neq 0$. However, we might still have $B^\sigma(\nu)=0$ for some of the permutations $\sigma \in S_\ell$. Indeed, if we write, for  $1\leq j\leq \ell$, 
$i_j:=\min\{i: N_j\geq\nu_i\}$, 
and set $U:=\{\sigma \in S_\ell: i_j\leq\sigma(j)\leq \ell\}$, then it is not hard to see that $B^\sigma(\nu)\neq 0$ if and only if $\sigma\in U$. (Note that $1\leq i_j\leq j$ and hence the identity permutation is clearly in $U$). All the polynomials $B^\sigma(\nu)$ for $\sigma \in U$ have the same lowest degree term (namely $(-\theta)^w$ with $w:=\sum_{i=1}^\ell \frac{q^{\nu_i}-1}{q-1}$). A simple counting argument gives that 
\[
|U|=\prod_{j=2}^\ell (j-i_j+1),
\]
and hence the size of $U$ is not divisible by $p$ since none of its factors is. (We are utilizing $\ell\leq q-1$ in this step). Thus the lowest degree term of $\sum_{\sigma \in U}B^\sigma(\nu)$ has nonvanishing coefficient $|U| (-1)^w$, and it follows that $\sum_{\sigma \in S_\ell}B^\sigma(\nu)=\sum_{\sigma \in U}B^\sigma(\nu)$ doesn't vanish whenever $N_i\geq \nu_i$ for all $i$, completing the proof. 
\end{proof}
\begin{remark}
 We would like to point out that there is a subtle yet important difference between the action on $S_\ell$ on the polynomials $P(\nu)$ and on $B(\nu$). The former is in fact an action on all of $Y$ and defines  an endomorphism; in particular we have $\sigma(0)=0$. On the other hand, the action of $S_\ell$ appearing in \eqref{mainf} is only a permutation action of certain parameters in the definition of $B(\nu)$ and doesn't extend to all of $A$. As an explicit example, if we take $\nu=\{2,1\}$ and $\{N_1,N_2\}=\{3,1\}$ then $B(\nu)=(\theta^{q^{3}}-\theta)(\theta^{q^3}-\theta^q)(\theta^{q}-\theta)$ but $B^\sigma(\nu)=0$ with $\sigma=(1\ 2)$.
\end{remark}

\section{Examples}\label{examples}

We look more closely at some concrete cases of Theorem \ref{modularmain}, and in particular of Example \ref{ex1}. We work with $q=3$,  $\ell=1$ and $N_1=1$; thus yielding the eigensystem $\frakp^4$. Working out the conditions in \cite[Theorem 1.3]{Pet}, we have the infinite family of forms with that eigensystem which is given for $n \geq 1$ by
\[
f_{18n+4,4}=\sum_{a\in A_+}a^{18n}G_4(t_a), 
\]
where the Goss polynomial $G_4$ is given explicitly by $G_4(X)=X^4+\frac{X^2}{\theta^3-\theta}$. In particular, we see that we have such eigenforms with $A$-expansion in the set of weights $\{22,40,58,76, \dots\}$. In addition, direct computations on lower weights reveal normalized (i.e. leading coefficient $a_2=1$) double-cuspidal eigenforms $\phi_{12}$ and $\phi_{20}$ both with eigensystem $\frakp^4$ in weights $12$ and $20$ respectively. Neither $\phi_{12}$ nor $\phi_{20}$ possess an $A$-expansion. Nonetheless, all of these forms fall under the scope of Example \ref{ex1}(1), and as the $t$-expansions illustrate, they indeed \emph{all} have $a_4=\theta^3-\theta$ (in $\F_3$) and $a_{10}=a_{28}=a_{82}=0=a_{1+3^i}$ for all $i\geq 2$. Explicitly we have (with $g:=g_{q-1}=g_2$)

\begin{align*}
\phi_{12}=h^2 g^2&=
t^2 + (\theta^3 + 2\theta) t^4 + (\theta^6 + \theta^4 + \theta^2 + 2) t^6 + (\theta^9 + 2 \theta^3) t ^{12} + t ^{14} \\
& +    (2 \theta^3 +\theta) t ^{16}+ 2 t^{18} + (2 \theta^9 + 2 \theta^3 + 2\theta) t^{20}+ (2 \theta^{12} + \theta^{10} +     2 \theta^6 + \theta^2) t^{22}\\
& + (2 \theta^{15} + 2 \theta^{13} + 2 \theta^{11} + \theta^9 + \theta^7 + \theta^5 + \theta^3 +
    2\theta) t^{24} + t^{26} \\
&+ (2 \theta^{18} + \theta^{12} + \theta^{10} + 2 \theta^4 + 2) t^{30} + O(t^{32}).
\end{align*}

\begin{align*}
\phi_{20}=&t^2 + (\theta^3 + 2\theta) t^4 + (\theta^6 + \theta^4 + \theta^2 + 2) t^6 + (2 \theta^9 +\theta) t^8 + (2 \theta^9 + 
    2 \theta^3 + 2\theta) t^{12}\\
& + (\theta^{18} + \theta^{12} + 2 \theta^4 + 2 \theta^2 + 1) t^{14} + (2 \theta^3 + 
    \theta) t^{16} + (2 \theta^{18} + \theta^{12} + \theta^{10} + 2 \theta^4 + 2) t^{18}\\& + (\theta^{21} + 2 \theta^{19} + 2 \theta^{15} 
    + \theta^{11} + \theta^9 + \theta^7 + 2 \theta^5 + 2 \theta^3) t^{20} + (\theta^{12} + 2 \theta^{10} + 2 \theta^6 + 
    \theta^4) t^{22}\\& + (\theta^{15} + \theta^{13} + \theta^{11} + 2 \theta^9 + 2 \theta^7 + 2 \theta^5 + 2 \theta^3 +\theta) t^{24} + 
    (\theta^{18} + \theta^{10} + \theta^2 + 1) t^{26} \\&+ (2 \theta^{12} + \theta^{10} + \theta^4 + 2 \theta^2 + 2) t^{30} + 
   O(t^{32}).
\end{align*}

Set $\phi_{22}=(\theta^3-\theta)f_{22,4}$. We have $\phi_{22}=h^2 g^7-(\theta^3-\theta)h^4g^3$ and 

\begin{align*}
\phi_{22}&=
t^2 + (\theta^3 + 2\theta) t^4 + 2 t^6 + (\theta^9 + 2\theta) t^8 + (2 \theta^9 +\theta) t^{12} + (\theta^{18} + 
    \theta^{12} + 2 \theta^4 + 2 \theta^2 + 1) t^{14}\\
& + (2 \theta^{21} + \theta^{19} + 2 \theta^{13} + \theta^{11} + 2 \theta^5 + 
    T) t^{16} + (2 \theta^{18} + 2 \theta^{10} + \theta^6 + \theta^4 + 2) t^{18}\\
& + (2 \theta^{21} + \theta^{19} + 2 \theta^{13} +
    \theta^{11} + \theta^9 + 2 \theta^5 + 2 \theta^3 +\theta) t^{20}\\
& + (2 \theta^{24} + 2 \theta^{22} + 2 \theta^{20} + 2 \theta^{16} + 
    2 \theta^{14} + 2 \theta^{10} + 2 \theta^8 + 2 \theta^4 + 2 \theta^2) t^{22} + t^{26} + 2 t^{30}+O(t^{32}).
\end{align*}

\begin{remark}
It is instructive to compare the previous examples with the expansion of $\Delta=\sum a_is^i \in M_{8,0}^2$  in powers of $s=t^2$ over $\F_3$. As alluded to in the Introduction, Gekeler's general results \cite{Gekeler88, Gekeler99} imply  the following  for $a_i\in \F_3[\theta]$: (i) $a_i\neq 0$ implies $i\equiv 0,1 \pmod 3$, (ii) $\deg(a_{1+i})\leq i$, and (iii) $\deg(a_{1+i})=i$ if and only $i\equiv 0,3 \pmod {9}$. The first few instances of those results may indeed be verified from the initial terms of the expansion of $\Delta$ below. On the other hand, it is clear that those special properties of $\Delta$ are not shared by any of the forms $\phi_{12}$, $\phi_{20}$ and $\phi_{22}$ above.

\begin{align*}
\Delta&=t^2 + 2t^6 + (\theta^3 + 2\theta)t^8 + t^{14} + 2t^{18} + (2\theta^9 + 2\theta^3 + 2\theta)t^{20} +    (\theta^9 + 2\theta^3)t^{24} \\
&+ (2\theta^{12} + \theta^{10} + \theta^6 + 2\theta^4 + 1)t^{26} + 2t^{30}+O(t^{32})
\end{align*}

\end{remark}



\end{document}